\documentclass[a4paper,12pt]{amsart}
\usepackage{amsmath,amsthm,amssymb}
\usepackage{hyperref}

\allowdisplaybreaks[1]

\numberwithin{equation}{section}
\newtheorem{theorem}{Theorem}[section]
\newtheorem{lemma}[theorem]{Lemma}
\newtheorem{proposition}[theorem]{Proposition}

\theoremstyle{remark}

\newtheorem{example}{Example}[section]
\newtheorem{definition}{Definition}[section]

\newcommand{\dl}{\delta}
\newcommand{\e}{\epsilon}

\newcommand{\ity}{\infty}
\newcommand{\C}{\mathbb{C}}

\newcommand{\ti}{\widetilde}

\newcommand{\ta}{\theta}
\newcommand{\io}{\iota}

\newcommand{\N}{\mathbb{N}}

\newcommand{\al}{\alpha}

\begin{document}

\title[attractors and chain recurrence  in noncompact space]
{attractors and chain recurrence in noncompact space  for semigroup of continuous maps}

\author[S. Kumar]{Sanjay Kumar}
\address{Sanjay Kumar\\Department of Mathematics\\ Deen Dayal Upadhyaya College\\University of Delhi\\Sector-3, Dwarka\\Delhi--110078, India}
\email{skpant@ddu.du.ac.in }

\author[K. Lalwani]{Kushal Lalwani \textsuperscript{*}}
\address{Kushal Lalwani\\Department of Mathematics\\ University of Delhi\\Delhi--110007, India}
\email{lalwani.kushal@gmail.com }

\thanks{\noindent \textsuperscript{*}Communicating author}


\begin{abstract}
We shall address the alternative definition of chain recurrent  set for the action of a  semigroup of continuous self maps, given by M. Hurley \cite {mh} in noncompact space. Following this, we shall address the characterization of chain recurrence in terms of attractors given by C. Conley in \cite {conley}. 
\end{abstract}

\maketitle

\section{Introduction}
A {\it continuous semigroup}  is a set of (non-identity) continuous self maps, of a topological space $X$,  which are closed under the composition. A semigroup $G$ is said to be generated by a family $\{g_{\al}\}_{\al}$ of  continuous self maps of a topological space $X$  if every element of $G$ can be expressed as compositions of iterations of the elements of  $\{g_{\al}\}_{\al}$. We denote this by $G=<g_{\al}>_{\al}$. The space $X$ is assumed to be Hausdorff and first countable.

M. Hurley had shown that chain recurrent set for a semiflow is the complement of the union of the set $B(A) \setminus A$, as $A$ varies over the collection of attractors and $B(A)$ denotes the basin of attraction. This concept for flows (continuous action of a group) on a compact space was originally introduced by Conley \cite {conley}. Later, Hurley in \cite {mh} extended this characterization for semiflows (continuous action of a semigroup) without the assumption of compactness. The main aim  of this paper is to study the notion of attractors and chain recurrence in the context of a continuous semigroup such that the classical case turns out as a special case of this generalization. In \cite {kl1,kl2}, we have studied the chain recurrent set and established the notion of attractors for semigroup of continuous maps. Also, we have extended the characterization of chain recurrent set in terms of attractors for compact metric spaces.  Here, we shall follow the treatment of Hurley for noncompact metric spaces and see how far this characterization applies in this setting of semigroup of continuous maps.

%
%
%

\section{Chain Recurrence in Noncompact Space} \label{secnat}

In \cite{mh1,mh2,mh}, Hurley generalized the definition of chain recurrence for noncompact spaces and expressed it in terms of attractors, analogous to the work of C. Conley.
The equivalence of the two definitions follows from the following theorems of Hurley. We will see the counterpart of this theorem in the context of action of a semigroup of continuous maps in Theorem \ref{atth}.

\begin{theorem} \cite{mh}
If $(X,d)$ is a metric space and $f: X \to X$ is  continuous, then the  chain recurrent set ${\sf{CR}}(f)$, is the complement of the union of sets $B(A) \setminus A$ as $A$ varies over the collection of attractors  of $f$:
$$X \setminus {\sf{CR}}(f)=\bigcup_A[B(A) \setminus A]$$
 where $B(A)$ denotes the basin of attraction of A.
\end{theorem}

\begin{theorem}\cite {mh}
Suppose that $(X,d)$ is a metric space and $\Phi: [0,\ity)\times X \to X$ is a continuous flow. The chain recurrent set of $\Phi$, denoted by ${\sf{CR}}(\Phi)$, is the complement of the union of sets $B(A) \setminus A$ as $A$ varies over the collection of attractors  of $\Phi$ and where $B(A)$ denotes the basin of attraction of A:
$$X \setminus {\sf{CR}}(\Phi)=\bigcup_A[B(A) \setminus A].$$
\end{theorem}

Following Hurley \cite {mh}, we introduce the notion of a chain recurrent set  for  a continuous semigroup in case of a noncompact space by using strictly positive continuous functions as follows: Let $P(X)$ denote the set of positive continuous functions on $X$, that is, $P(X):=\{\e:X\to (0,\ity): \e {\rm\ is\  continuous}\}$.

\begin{definition} \index{chain! for noncompact spaces}
Let $G$ be a semigroup of continuous self maps defined on a metric space $(X,d)$. Let $a,b \in X$, $g\in G$ and $\e \in P(X)$ be  given. An $(\e,g)${\it-chain} from $a$ to $b$ means a finite sequence $(a=x_1,\dots, x_{n+1}=b;g_1, \dots, g_n)$, where $x_i\in X$ and $g_i \in \widehat G$ such that $d(g_i\circ g(x_i),x_{i+1})< \e(g_i\circ g(x_i))$, for each $i=1, \dots ,n$.
\end{definition}

\begin{definition} \index{chain recurrent point! for noncompact spaces}
Let $G$ be a semigroup on a metric space $(X,d)$. A point $x\in X$ is called a {\it chain recurrent point} for $G$ if for every $\e\in P(X)$ and every $g\in G$ there exists an $(\e, g)$-chain from $x$ to itself. The set of all chain recurrent points for $G$ is denoted by ${\sf{CR}}(G)$.
\end{definition}
Following two lemmas have been proved by Hurley in \cite {mh}.

\begin{lemma}\label{hlm}
Let $(X,d_X)$ and $(Y,d_Y)$ be two  metric spaces and $f: X \to Y$ is a continuous map. Then, for each $\e \in P(Y)$ there is a $\dl \in P(X)$ such that
$$d_Y(f(x),f(y))< \e (f(x)) \quad {\rm{whenever}} \quad d_X(x,y)<\dl(x). $$
\end{lemma}

\begin{lemma}\label{mhlm}
If $\e \in P(X)$ there is a $\dl \in P(X)$ with $\dl <\e/2$ and satisfying $\e(y)>\e(x)/2$ whenever $d(y,x)<\dl(x)$.
\end{lemma}

\begin{theorem}
If $G$ is abelian then the chain recurrent set ${\sf{CR}}(G)$ is invariant under $G$.
\end{theorem}

\begin{proof}
Let $x \in {\sf{CR}}(G)$ and $g$ be any generator of $G$.  Let $\e  \in P(X)$ be given and some $h \in G$. We shall construct an $(\e,h)$-chain from $g(x)$ to itself and it will follow that $g(x)\in {\sf{CR}}(G)$.

Since $g$ is continuous, by Lemma \ref{hlm}, there exists a $\dl \in P(X)$ such that $\dl \leq \e$ and  if $d(x,y)< \dl(x) \ {\rm then}\  d(g(x),g(y))$  $< \e(g(x)).$  Also $x \in {\sf{CR}}(G)$ implies there is a $(\dl, g\circ h)$-chain $(x=x_1,\dots , x_{n+1}=x;h_1, \dots ,h_n)$ from $x$ to itself. That is, for each $i=1, \dots ,n$, we have,
$$d(h_i \circ g \circ h(x_i),x_{i+1})<\dl(h_i \circ g \circ h(x_i)) \leq \e(h_i \circ g \circ h(x_i)).$$

In particular, as $G$ is abelian, we have,
\begin{equation} \notag
\begin{split}
d(h_1 \circ h \circ g(x_1),x_2)& =d(h_1 \circ g \circ h(x_1),x_2)\\
& <\dl(h_1 \circ g \circ h(x_1))\\
& \leq \e(h_1 \circ g \circ h(x_1)).
\end{split}
\end{equation}

Also, by the continuity of $g$, we have, 
$d(h_n \circ g \circ h(x_n),x)<\dl(h_n \circ g \circ h(x_n))$  implies $d(g \circ h_n \circ g \circ h(x_n),g(x))<\e(g \circ h_n \circ g \circ h(x_n)).$

Hence, $(g(x),x_2,\dots , x_n,g(x);h_1,h_2 \circ g, \dots,h_{n-1}\circ g ,g \circ h_n\circ g)$ is  an $(\e,h)$-chain from $g(x)$ to itself.
\end{proof}

\begin{theorem}
The set of all chain recurrent points for $G$ is a closed subset of $X$.
\end{theorem}

\begin{proof}
Suppose $a\in X$ is a limit point of ${\sf{CR}}(G)$. Let $\e\in P(X)$ and a $g\in G$. Since $g$ is continuous, by Lemma \ref{hlm}, there exists a $\dl_1 \in P(X)$ such that  if $d(x,a)< \dl_1(a) \ {\rm then}\  d(g(x),g(a))$  $< \e(g(a))$.

Also, by Lemma \ref{mhlm}, we have, for $\frac{\e}{2}\in P(X)$ there is an $\e\rq{}\in P(X)$ with $\e\rq{} <\frac{1}{2}\frac{\e}{2}$ and satisfies $\e(y)/2>\e(x)/4$ whenever $d(y,x)<\e\rq{}(x)$. Again, by Lemma \ref{mhlm}, we have, for $\e\rq{}\in P(X)$ there is a $\dl_2 \in P(X)$ with $\dl_2 <\frac{\e\rq{}}{2}$ and satisfies $\e\rq{}(y)/>\e\rq{}(x)/2$ whenever $d(y,x)<\dl_2(x)$.

Define a function $\dl \in P(X)$ as $\dl(x)=\min\{\dl_1(x), \dl_2(x)\}$. Since $a$ is a limit point of ${\sf{CR}}(G)$, there is $x\in {\sf{CR}}(G)$ such that $d(x,a)< \dl(a)$. Then $d(g(x),g(a))<\e(g(a))$ and hence $(a,g(x); identity)$ is an $(\e,g)$-chain from $a$ to $g(x)$.

Now, for $x\in {\sf{CR}}(G)$, there is a $(\dl,g^2)$-chain $(x=x_1,\dots, x_{n+1}=x;g_1, \dots, g_n)$ from $x$ to itself. Then $(g(x),x_2, \ldots, x_n;g_1,g_2g, \dots, g_{n-1}g)$ is an $(\e,g)$-chain from $g(x)$ to $x_n$. In particular, we have,
\begin{equation} \notag
\begin{split}
d(g_ng^2(x_n),x)&<\dl(g_ng^2(x_n))\\
&\leq  \dl_2(g_ng^2(x_n))\\
&<  \frac{\e\rq{}}{2}(g_ng^2(x_n))\\
&<  \e\rq{}(x)\\
&<\frac{\e}{2}(g_ng^2(x_n)).
\end{split}
\end{equation}

Since
\begin{equation}\notag
\begin{split}
d(g_ng^2(x_n),a)&\leq d(g_ng^2(x_n),x)+ d(x,a)\\
&<  \dl_2(g_ng^2(x_n))+\dl(a)\\
&<  \e\rq{}(x)+ \e\rq{}(x)\\
&<\e(g_ng^2(x_n)),
\end{split}
\end{equation}

we have, $(x_n,a;g_ng)$ is an $(\e,g)$-chain from $x_n$ to $a$.
By transitivity via concatenating these $(\e,g)$-chains, there is an $(\e,g)$-chain from $a$ to itself.
\end{proof}

\begin{theorem}
Let $(X,G)$ and $(Y,\ti{G})$ be two dynamical systems on the metric spaces $(X,d_X)$ and $(Y,d_Y)$. If a homeomorphism $\rho : X \to Y$ is a topological conjugacy then $\rho({\sf{CR}}(G))={\sf{CR}}(\ti{G})$.
\end{theorem}

\begin{proof}
Let $x \in {\sf{CR}}(G)$.  Let $\e \in P(Y)$ be given and some $\ti g \in \ti G$. First we shall construct an $(\e,\ti g)$-chain from $y=\rho(x)$ to itself. Let $g \in G$ be such that $\rho \circ g=\ti g \circ \rho$.

Since $\rho : X \to Y$ is continuous, by Lemma \ref{hlm}, there exists a $\dl \in P(X)$ such that if $d_X(x_1,x_2)< \dl(x_1)$ then $d_Y(\rho(x_1),\rho(x_2))< \e(\rho(x_1))$, for $x_1,x_2 \in X$.  Also $x \in {\sf{CR}}(G)$ implies there is a $(\dl, g)$-chain $(x=x_1,\dots , x_{n+1}=x;g_1, \dots ,g_n)$ from $x$ to itself. That is, for each $i=1, \dots ,n$, we have,
$$d_X(g_i\circ g(x_i),x_{i+1})<\dl(g_i\circ g(x_i)) .$$

For  each $i=1, \dots ,n+1$, let $y_i=\rho(x_i)$ and $\rho \circ g_i=\ti g_i \circ \rho$. Since 
$$d_X(g_i\circ g(x_i),x_{i+1})<\dl(g_i\circ g(x_i)) ,$$
we have,
\begin{equation} \notag
\begin{split}
d_Y(\ti g_i\circ \ti g(y_i),y_{i+1})&=d_Y(\rho \circ g_i\circ g(x_i),\rho(x_{i+1}))\\
& <\e(\rho \circ g_i\circ g(x_i))\\
& = \e(\ti g_i\circ \ti g(y_i)).
\end{split}
\end{equation}

Thus $(\rho(x)=y_1,\dots , y_{n+1}=\rho(x); \ti g_1, \dots ,\ti g_n)$ is an $(\e,\ti g)$-chain  from $\rho(x)$ to itself and hence $\rho(x)\in {\sf{CR}}(\ti{G})$.  Therefore $\rho({\sf{CR}}(G)) \subset {\sf{CR}}(\ti{G})$.

Conversely, let $y \in {\sf{CR}}(\ti{G})$. Since $\rho^{-1}$ is a conjugacy from $Y$ to $X$, we have, $y \in \rho({\sf{CR}}(G))$. Thus, $\rho({\sf{CR}}(G))={\sf{CR}}(\ti{G})$.
\end{proof}

\section{Conley\rq{}s Theorem for semigroup of Continuous Maps}

This section consists of systematic  investigation to reproduce the concept of  attractors for semigroup and provide the alternative definition of a chain recurrent set. In \cite {kl2}, we have introduced the notion of an attractor for the semigroup of continuous maps as follows:
\begin{definition} \label {tr}
A nonempty open subset $U$ of $X$ is said to be a {\it trapping region} for $G$ if  there exists an $h\in G$ such that
  for the set
$$\widetilde U := \{fh(x): x\in U\ {\rm and \ } f \in \widehat G \}$$

we have,  $cl(\widetilde U) \subset U$.
\end{definition}

\begin{definition} \label {at}
The {\it attractor} for $G$ determined by  a {\it trapping region} $U$ for $G$ is defined by
\begin{equation} \notag
\begin{split}
A:=&\{x\in X : {\rm for\ every\ open \ subset}\ V\ {\rm of}\  X\ {\rm containing}\  x,\ V \cap f_{k}(h(U))\ne \emptyset \\  
&\ {\rm for\  infinitely\  many}\  k \in \N, {\rm where}\ (f_{k})_k\ {\rm is\ some}\ {\rm unbounded\ sequence\  in\  } G \},
\end{split}
\end{equation}
where, $h\in G$ is as in the Definition \ref {tr}.
\end{definition}

The notion of unbounded sequence for the continuous semigroup was introduced in \cite {kl} as follows:
A sequence of functions $(f_{k})_{k\in \N}$ in $G=\langle g_{\al}\rangle _{\al}$ is said to be {\it{unbounded}} if there is 
\begin{enumerate}
\item a sequence $(n_k)_k$ of natural numbers with $n_k\to \ity$ as $k\to \ity$; and
\item  a  generator $g_{\al_{_0}}\in \{g_{\al}\}_{\al}$ such that  each $f_{k}$ consists of exactly $n_k$ iterates of  $g_{\al_{_0}}$,  that is, $f_{k}=h_{n_k+1}\circ g_{\al_{_0}}\circ h_{n_k} \circ g_{\al_{_0}} \circ h_{n_k-1} \circ \dots \circ h_2\circ g_{\al_{_0}} \circ h_1$, where each $h_i \in \widehat G= G\cup \{identity\}$ and the functions $h_i$ are independent of $g_{\al_{_0}}$.
\end{enumerate}

Note that the unboundedness of a sequence in $G$ is not with respect to some metric on $G$. The term unbounded refers to the unboundedness of the sequence $(n_k)$, the number of iterates of a generator of $G$. Recall that in case of a discrete or continuous dynamical system $(X, T, \Phi)$, any unbounded sequence corresponds to the iterates of $\Phi^1$.

\begin{definition}
The {\it basin of attraction} of an attractor $A$ for $G$ determined by  a  trapping region $U$  is defined by
$$B(A):=\{x \in X : f(x) \in U {\rm \ for\  some\ } f \in G \}.$$
\end{definition}

It is clear that  $B(A)$ contains $U$. 

The following lemma has been given by Hurley in \cite {mh}.

\begin{lemma}\label{mhlm}
If $\e \in P(X)$ there is an $\dl \in P(X)$ with $\dl <\e/2$ and satisfying $\e(y)>\e(x)/2$ whenever $d(y,x)<\dl(x)$.
\end{lemma}

We notice the following characteristics of attractors in \cite {kl2}

\begin{proposition}
The attractor $A$ determined by a trapping region $U$ is invariant under $G$.
\end{proposition}

\begin{proposition}
The attractor $A$ determined by a trapping region $U$ is a closed set.
\end{proposition}

\begin{proposition} \label {atp1}
Let $A$ be the attractor determined by a trapping region $U$. Then $A$ is contained in $U$.
\end{proposition}

\begin{proof}
Since $U$ is a trapping region there exists an $h\in  G$ such that  ${\sf{cl}}(\widetilde U) \subset U$. Define $\eta \in P(X)$  by
$$\eta(x)=\frac{1}{2}\left(d(x,{\sf{cl}}(\widetilde U))+ d(x,X\setminus U)\right).$$
Thus, if $y \in {\sf{cl}}(\widetilde U)$, $z \in X$, and $d(y,z)<\eta(y)$, then $z\in U$.

Let $x \in A$. By Lemma \ref{mhlm}, there exists a $\dl \in P(X)$ corresponding to $\eta$. Then $B(x,\dl(x))\cap f_{k}(h(U)) \ne \emptyset$  for  infinitely  many  $k \in \N$,  where  $(f_{k})_k$ is some unbounded sequence  in  $ G$.  Also,
$$ f_{k}(h(U)) \subset \widetilde U \subset {\sf{cl}}(\widetilde U).$$

Therefore, $B(x,\dl(x))\cap {\sf{cl}}(\widetilde U) \ne \emptyset$. Let $y \in B(x,\dl(x))\cap {\sf{cl}}(\widetilde U)$. Since
$$d(x,y)<\dl (x) < \eta(x)/2 <\eta(y) ,$$
we have, $x\in U$.
\end{proof}

\begin{theorem}  \label{atth}
Let $(X,d)$ be a metric space and $G$ be an abelian semigroup. The chain recurrent set of $G$ is the complement of the union of sets $B(A) \setminus A$ as $A$ varies over the collection of attractors  of $G$:
$$X \setminus {\sf{CR}}(G)=\bigcup_A[B(A) \setminus A].$$
\end{theorem}

\begin{proof}
Let $A$ be an attractor of $G$ and $p\in B(A) \setminus A$. Let $U$ be a trapping region which determines $A$, and $h\in  G$ is such that  ${\sf{cl}}(\widetilde U) \subset U$. As in Proposition \ref {atp1}, there is $\eta \in P(X)$ such  that if $y \in {\sf{cl}}(\widetilde U)$, $z \in X$, and $d(y,z)<\eta(y)$, then $z\in U$. We can take $\eta$ to be bounded above by 1.

For $p\in B(A)$, we have, $f(p) \in U$ for some $f \in G$. Let $(f_{k})$ be an unbounded sequence  in $G$. Since $U$ is a trapping region $ f_{k}hf(p)\in  U.$ Let $\hat f_{k}$ denotes $f_{k}hf $. 

Now, if $p \in {\sf{CR}}(G)$, then for each $k,m \in \N$,  there exists an $(\eta/m,\hat f_{k}h)$-chain $(p=x_1, \dots , x_{l_k+1}=p;h_1, \dots ,h_{l_k})$ from $p$ to itself. Since $G$ is abelian and $\hat f_{k}(p)\in  U$, we have,
$$h_1\hat f_{k}h(x_1) \in h_1h(U) \subset  \widetilde U \subset {\sf{cl}}(\widetilde U).$$
Since 
$$d(h_1\hat f_{k}h(x_1),x_2)<\frac{1}{m}\eta(h_1\hat f_{k}h(x_1)) \le \eta(h_1\hat f_{k}h(x_1)) ,$$
we have, $x_2 \in  U$. Similarly,
$$h_2\hat f_{k}h(x_2) \in  \widetilde U \subset {\sf{cl}}(\widetilde U)$$
and $d(h_2\hat f_{k}h(x_2),x_3)<\frac{1}{m}\eta(h_2\hat f_{k}h(x_2)) \le \eta(h_2\hat f_{k}h(x_2))$ gives $x_3 \in U$. Following the previous arguments, we have $x_i \in U$ for every $i=1, \dots ,{l_k}$. Since 
$$d(h_{l_k}\hat f_{k}h(x_{l_k}),p)<\frac{1}{m}\eta(h_{l_k}\hat f_{k}h(x_{l_k})) \le \frac{1}{m},$$
we have,
$$B_{\frac{1}{m}}(p)\cap h_{l_k}\hat f_{k}h((U)) \ne \emptyset.$$

Now, the sequence of functions defined by
$$ \bar{f}_{k}= h_{l_k}\hat f_{k} \ ,$$
is an unbounded sequence in $G$. Also, for any open subset $V$ of $X$ containing $p$, there is $m \in \N$ such that $B_{\frac{1}{m}}(p) \subset V$ and   an unbounded sequence $\bar{f}_{k}$ in $G$ corresponding to the $m$. Thus,
$$V \cap \bar{f}_{k}(h(U)) \supset B_{\frac{1}{m}}(p)\cap \bar{f}_{k}(h(U)) \ne \emptyset ,$$
for all $k$. Thus $p \in A$, which is a contradiction. 
Therefore, if $p\in B(A) \setminus A$ then $p \notin {\sf{CR}}(G)$.

Conversely, let $p \notin {\sf{CR}}(G)$. Then there exists an $\e \in P(X)$ and $h_1 \in  G$ such that there is no $(\e,h_1)$-chain from $p$ to itself. Consider the set defined by
$$U:=\{x\in X:{\rm there\ is\ an}\ (\e,h_1){\rm -chain\ from\ } p\ {\rm to\ }x \}.$$
Then $p \notin U$ and  
$$\widetilde U := \{fh_1(x): x\in U\ {\rm and \ } f \in \widehat G \}.$$

Also, $U$ is an open subset of $X$. Let $x\in U$, there exists an $(\e,h_1)$-chain $(p=x_1, \dots , x_{n+1}$ $=x;\ f_1, \dots ,f_n)$ from $p$ to $x$. Now, $x\in B(f_nh_1(x_n), \e(f_nh_1(x_n))) \subset U$. If $y \in  B(f_nh_1(x_n),$ $ \e(f_n h_1(x_n)))$ then $(p=x_1, \dots , x_{n+1}=y;f_1, \dots ,f_n)$ from $p$ to $x$ is  an $(\e,h_1)$-chain .

Let $y \in {\sf{cl}}(\widetilde U)$.  By Lemma \ref{mhlm}, there exists a $\dl \in P(X)$ corresponding to $\e$. Then for some $x \in U$ and $f \in \widehat G$, we have, $f h_1(x) \in B(y, \dl(y))$. Also,
$$d(y,f h_1(x))< \dl(y)< \frac{1}{2}\e(y) <\e(f h_1(x)).$$
Thus, $(x,y;f)$ is an  $(\e ,h_1)$-chain from $x$ to $y$. By transitivity, we can produce an $(\e,h_1)$-chain from $p$ to $y$.
 Hence ${\sf{cl}}(\widetilde U)) \subset U$. Therefore, $U$ is a trapping region for $G$.

Let $A$ be an attractor determined by $U$. Since $p \notin U$ and $A \subset U$, it follows that $p \notin A$.

Moreover, $(p,h_1(p); identity)$ is an $(\e,h_1)$-chain from $p$ to $h_1(p)$. Hence, $h_1(p)\in U$ implies $p \in B(A)$. Therefore, $p \in B(A) \setminus A$.
\end{proof}

The following example will illustrate an application of the above theorem.
\begin{example}
Consider the  polynomial mappings of the complex plane, given by
$$g_n : \C \to \C$$
$$g_n(z)=z^n.$$

We shall identify all the trapping regions and attractors for the dynamics of $G$ on the phase space $\C$. Evidently, the empty set and the phase space itself are trapping regions for $G$. Moreover, in this case, the corresponding attractor and basin of attraction coincide with the trapping regions respectively. Thus, we have, $B(A) \setminus A=\emptyset$ for each attractor.

Next, we shall classify the attractors and basin of attraction for a proper nonempty trapping region of the phase space $\C$ in the following two types. Since the action of elements of the semigroup results on the set $\{z\in \C : |z|>1\}$  in radial expansion and  on  the set $\{z\in \C : |z|<1\}$  in radial contraction. Consequently, any proper subset of the phase space containing the unit circle cannot be a trapping region. Therefore, the attractors and basin of attraction for any trapping region correspond to one of the following type.

{\it Type I:} For each $r\in (0,1)$, the set $U_r=\{z\in \C :|z| < r\}$ is a trapping region for $G$. Since $z\in U_r$ implies $|g(z)| \leq |z|^2 <r$ for every $g\in G$. In particular, for $\widetilde U_r=\{g\circ g_{_2}(z): z\in U_r {\rm \ and\ } g \in \widehat G\}$, we have, ${\sf{cl}}(\widetilde U_r) \subset U_r$. Here, for each $r\in (0,1)$,  the attractor $A_r$ is $\{0\}$ and the basin of attraction $B(A_r)$ is the set $\{z\in \C : |z|<1\}$ for each $r$. Therefore, in this case, $B(A_r) \setminus A_r=\{z\in \C : 0<|z|<1\}$.

{\it Type II:} For each $R \in (1, \ity)$ ,  the set $U_R=\{z\in \C :|z| > R\}$ is a trapping region for $G$. Now,   for each $R \in (1, \ity)$ the attractor $A_R$ is the empty set and the basin of attraction $B(A_R)$ is the set $\{z\in \C : |z|>1\}$ for each $R$. Hence,  $B(A_R) \setminus A_R=\{z\in \C : |z|>1\}$.

By Theorem \ref{atth}, we have,
\begin{equation} \notag
\begin{split}
X \setminus {\sf{CR}}(G) &= \bigcup_A[B(A) \setminus A] \\
\qquad &=\emptyset \cup \{z\in \C : 0<|z|<1\} \cup \{z\in \C : |z|>1\} \\
&= \{z \in \C : |z|\neq 0, |z|\neq 1\}.
\end{split}
\end{equation}
Consequently, we have, ${\sf{CR}}(G)=\{z \in \C : |z|= 0 {\rm \ or \ } |z|= 1\}$.

Next, we shall find all the chain recurrent points for $G$  and show that ${\sf{CR}}(G)=\{z \in \C : |z|= 0 {\rm \ or \ } |z|= 1\}$.

As 0 is a fixed point of $G$, hence a chain recurrent point for $G$.

Let $z\in \{z\in \C : 0<|z|<1\}$, we shall show that $z$ is not a chain recurrent point for $G$. Let $m,M \in \N$ be sufficiently large, so that $1/m$ is small enough such that $|g_{_M}(z)-z|>\frac{1}{m}$. Let $\e \in P(\C)$ be bounded above by $1/m$. Then there is no $(\e,g_{_M})$-chain from $z$ to itself. Thus, $z$ is not a chain recurrent point for $G$.

Also, by similar argument, we see that $z\in \{z\in \C : |z|>1\}$ can not be a chain recurrent point for $G$.

Let $z_0 \in \C$ be such that $|z_0|=1$, that is, $z_0=\exp(\io \pi \ta_0)$.  Let $\e \in P(\C)$ and $n_0\in \N \setminus \{1\}$. We shall construct an $(\e, g{_{n_0}})$-chain from $z_0$ to itself. 

For rationals are dense in reals, we can pick a $z_1=\exp(\io \pi p_1/q_1) \in B_{\e}( g{_{n_0}}(z_0))=\{z:|g{_{n_0}}(z_0)-z|<\e(g{_{n_0}}(z_0))\}$, for some $p_1,q_1 \in \N $. Then, 
\begin{equation} \notag
\begin{split}
g_{2q_1}g{_{n_0}}(z_1) &=\exp(\io \pi p_1 n_0 2q_1/q_1)\\
&=1.
\end{split}
\end{equation}

Since the unit circle is compact, there is $\e_0=\min \{\e(z): |z|=1\}$. Let $p,q \in \N$ be such that $w=\exp(\io \pi p/q) \in B_{\e_0}(z_0)=\{z:|z_0-z|<\e_0\}$. There is an $N \in \N$ sufficiently large such that $z_2=\exp(\io \pi p/(qn_0N)) \in B_{\e}(1)=\{z:|1-z|<\e(1)\}$. Then $g_{_N}g_{n_0}(z_2)=w$ and $|w-z_0|<\e_0 \leq \e(w)$.

Therefore, $(z_0,z_1,z_2,z_0;identity,g_{2q_1},g_{_N} )$ is an $(\e, g{_{n_0}})$-chain from $z_0$ to itself. Thus, $z_0$ is a chain recurrent point for $G$ and, we have, ${\sf{CR}}(G)=\{z \in \C : |z|= 0 {\rm \ or \ } |z|= 1\}$.
\end{example}

\end{document}